\documentclass{amsart}
\usepackage{amsmath,amstext,amssymb,amsfonts,amscd,graphicx}
\newtheorem{thm}{Theorem}[section]
\newtheorem*{thma}{Theorem 1.5}

\newtheorem{lem}[thm]{Lemma}

\theoremstyle{definition}
\newtheorem{defn}[thm]{Definition}
\theoremstyle{remark}
\newtheorem{rem}[thm]{Remark}
\numberwithin{equation}{section}
\renewcommand{\l}{\lambda}
\newcommand{\norm}[1]{\left\Vert#1\right\Vert}

\newcommand{\To}{\longrightarrow}

\newcommand{\lr}[1]{\langle #1 \rangle}
\begin{document}

\title[Bilinear strichartz estimates for Schr\"{o}dinger operators]
{Bilinear Strichartz estimates for Schr\"{o}dinger operators in 2
dimensional compact manifolds with boundary and cubic NLS}%
\author{Jin-Cheng Jiang}
\address{INSTITUTE OF MATHEMATICS, ACADEMIA SINICA, TAIPEI, TAIWAN, 11529, R.O.C.}
\email{jiangjc@math.sinica.edu.tw}

\thanks{The author would like to thank Christopher Sogge for suggesting the problem and
numerous helpful discussions during this study. He would also like to thank Matthew Blair for
many helpful discussions. }
\maketitle

\begin{abstract}
In this paper, we establish bilinear and gradient bilinear Strichartz estimates for
Schr\"{o}dinger operators in 2 dimensional compact manifolds with boundary.
Using these estimates, we can infer the local well-posedness of cubic nonlinear Schr\"{o}dinger
equation in $H^s$ for every $s>\frac{2}{3}$ on such manifolds.
\end{abstract}

\section{Introduction and Results}
Let $(M,g)$ be a Riemannian manifold of dimension $n\geq2$.
Consider the Schr\"{o}dinger equation
\begin{equation}\label{E:Schrodinger}
D_tu+\Delta_gu=0,\;\;\;\;\;\;\;\;u(0,x)=f(x)
\end{equation}
where $\Delta_g$ denotes the Laplace-Beltrami operator on
manifold and $D_t=i^{-1}\partial_t$. Strichartz estimates are a
family of dispersive estimates on solutions $u(t,x):[0,T]\times
M\rightarrow \mathbb{C}$ which state
\begin{equation}\|u\|_{L^p([0,T];L^q(M))}\leq C\|f\|_{H^s(M)}
\end{equation}
 where $H^s$ denotes the $L^2$ Sobolev space over $M$, and
$2\leq p,q\leq\infty$ satisfies
$$\frac{2}{p}+\frac{n}{q}=\frac{n}{2}\;\;\;\;\;\;\;\;(n,p,q)\neq(2,2,\infty).$$

In Euclidean space, one can take $T=\infty$ and $s=0$; see for
example Strichartz~\cite{Str77}, Ginibre and Velo~\cite{GiV84}, Keel
and Tao~\cite{KeT98} and references therein. Such estimates have
been a key tool in the study of nonlinear Schr\"{o}dinger
equations. In the case of compact manifolds $(M,g)$ without boundary Burq,
G\'{e}rard and Tzvetkov~\cite{BGT04} proved the finite time scale
estimates (1.2) for the Schr\"{o}dinger operators with a loss of
derivatives $s=\frac{1}{p}$ in their
estimates when compared to the case of flat geometries.

In the case of compact manifolds with boundary, one considers Dirichlet or
Neumann boundary conditions in addition to (1.1)
$$u(t,x)|_{\partial M}=0\;({\rm Dirichlet}),\;\;\;{\rm or}\;\;\;\;N_x\cdot
\nabla u(t,x)|_{\partial M}=0\;({\rm Neumann})$$
where $N_x$ denotes the unit normal vector field to $\partial M$.
Here one excepts a further loss of derivatives due to Rayleigh
whispering galley modes. Recently, Anton ~\cite{Ant05} showed that
the estimates (1.2) hold on general manifolds with boundary if
$s>\frac{3}{2p}$ which arguments of~\cite{Ant05} work equally well for a manifold
without boundary equipped with a Lipschitz metric. Then Blair,
Smith and Sogge~\cite{BSS07} built
estimates (1.2) with a less loss of derivatives $s=\frac{4}{3p}$ in manifolds with boundary.

Write $u=e^{it\Delta}f$ as the solution of~\eqref{E:Schrodinger} with initial data $f$.
We consider bilinear estimates for the Schr\"{o}dinger operators
in compact manifolds of the form
\begin{equation}\label{E:bilinearStrione}
\|e^{it\Delta}fe^{it\Delta}g\|_{L^2([0,1]\times M)}\leq C({\rm
min}(\Lambda,\Gamma))^{s_0}\|f\|_{L^2(M)}\|g\|_{L^2(M)}
\end{equation}
where $\Lambda,\Gamma$ are large dyadic numbers, and $f,g$ are supposed to be spectrally localized on dyadic intervals
of order $\Lambda,\Gamma$ respectively, namely
$$\mathbb{I}_{\Lambda\leq\sqrt{-\Delta}\leq2\Lambda}(f)=f\;\;\;,\;\;\;\mathbb{I}_{\Gamma\leq\sqrt{-\Delta}
\leq2\Gamma}(g)=g.$$
Here $\mathbb{I}_{\Lambda\leq\sqrt{-\Delta}\leq2\Lambda}$ denotes the spectral projection operator $$\sum_{\Lambda\leq\Lambda_j\leq2\Lambda}E_jf=\sum_{\Lambda\leq\Lambda_j\leq2\Lambda}e_j\int_Mfe_j\;,$$while $\{\Lambda_j^2\}$ and $\{e_j\}$ are eigenvalues and corresponding eigenfunctions of $-\Delta_g$.
Such kind of estimates were established and used on
Schr\"{o}dinger equation on manifolds with flat metric; see
Klainerman-Machedon-Bourgain-Tataru~\cite{Kla96}, Bourgain~\cite{Bo93}
and Tao~\cite{Tao01} and reference therein .  Then Burq, G\'{e}rard and Tzvetkov
~\cite{BGT05} established the bilinear estimates in sphere and Zoll
surfaces with $s_0>\frac{1}{4}$. In the cases of
sphere and Zoll surfaces~\cite{BGT05} ,
due to the good locations of eigenvalues for the Laplacian, the bilinear
Strichartz estimates are reduced to bilinear spectral cluster estimates. For general manifolds,
our poor knowledge of spectrums does not allow us to use the same technique. One of our main results here
is showing that by considering the endpoint of admissible pairs for the Schr\"{o}dinger operator and
using the parametrix construction, we can get the bilinear Strichartz estimates for general 2
dimensional manifolds, though the estimates are not known to be sharp.

Consider Strichartz estimates on manifolds with boundary
obtained by Blair, Smith and Sogge~\cite{BSS07}.
When $n=2$, $(p,q)=(4,4)$ is admissible, so we have
\begin{equation*}
\|e^{it\Delta}f\|_{L^4([0,1]\times M)}\leq C\|f\|_{H^{1/3}(M)}.
\end{equation*}
Using Littlewood-Paley theory,
let $f_{\Lambda}=\mathbb{I}_{\Lambda\leq\sqrt{-\Delta}\leq2\Lambda}(f)$,
this is equivalent to say
$\|e^{it\Delta}f_{\Lambda}\|_{L^4([0,1]\times M)}\leq C\Lambda^{1/3}\|f_{\Lambda}\|_{L^2(M)}$
holds for all dyadic
number $\Lambda$, which
is implied by bilinear estimates~\eqref{E:bilinearStrione} with $s_0=\frac{2}{3}$. However we
would establish the following estimates with $s_0>\frac{2}{3}$.\par

\begin{thm}\label{T:bilinearone}
Let $(M,g)$ be a 2 dimensional compact manifold with boundary. For
any $f,\;g\in L^2(M)$ satisfies
\begin{equation*}
\mathbb{I}_{\Lambda\leq\sqrt{-\Delta}\leq2\Lambda}(f)
=f\;\;\;\;\mathbb{I}_{\Gamma\leq\sqrt{-\Delta}\leq2\Gamma}(g)=g.
\end{equation*}
Then for any $s_0>\frac{2}{3}$, there exists a $C>0$ such that
\begin{equation}\label{E:bilinearStrionemine}
\|e^{it\Delta}fe^{it\Delta}g\|_{L^2([0,1]\times M)}\leq C({\rm
min}(\Lambda,\Gamma))^{s_0}\|f\|_{L^2(M)}\|g\|_{L^2(M)}.
\end{equation}
\end{thm}
\smallskip

\begin{rem}
Our proof of Theorem~\ref{T:bilinearone}  can be simplified to get the bilinear Strichartz
estimates with $s_0>\frac{1}{2}$ in 2 dimensional compact manifolds without boundary.
\end{rem}

For compact manifold with boundary, Anton~\cite{Ant06}
proved~\eqref{E:bilinearStrione} and the following
\begin{equation}\label{E:bilinearStritwo}
\|(\nabla e^{it\Delta}f)e^{it\Delta}g\|_{L^2([0,1]\times M)}\leq
C\Lambda({\rm min}(\Lambda,\Gamma))^{s_0}\|f\|_{L^2(M)}\|g\|_{L^2(M)}
\end{equation} with $s_0>\frac{1}{2}$ on three dimensional balls with Dirichlet boundary condition
and radial data. She used the same idea as~\cite{BGT05}, thanks again the good locations of
eigenvalues for the Laplacian in such setting.
Using~\eqref{E:bilinearStrione} and~\eqref{E:bilinearStritwo} with $s_0>\frac{1}{2}$ , she proved the local
well-posedness of cubic nonlinear Schr\"{o}dinger equation with Dirichlet boundary condition and
radial data in $H^s$ for every $s>\frac{1}{2}$ on three dimensional balls.
In order to build the corresponding
estimates in our case, we need more results from harmonic analysis besides the parametrix construction of
solutions for the free equation. There are two different cases. If the gradient operator is acting on the
solution has initial data being localized to the larger frequency, then we can exploit the boundedness of Riesz
transform (see~\cite{Shen05}) on $L^2(M)$ ,  then apply the H\"{o}rmander multiple theorem
(for manifold with boundary, see~\cite{Xu04}) to get the desired result. For the other case,
we make use of Xu's~\cite{Xu04} estimates for the gradient spectral cluster operators. Following
by an argument concerning the finite propagation speed of solutions to the wave equation (see for example~\cite{Sog02},~\cite{Xu04} ), then we can control the $L^2$ norm from the estimates of gradient spectral
cluster operators by a $L^{\infty}$ norm, thus return to the parametrix construction argument again.   \par

Our gradient bilinear Strichartz estimate is the following.
\begin{thm}\label{T:bilineartwo}
Let $(M,g)$ be a 2 dimensional compact manifold with boundary. For
any $f,\;g\in L^2(M)$ satisfies
\begin{equation*}
\mathbb{I}_{\Lambda\leq\sqrt{-\Delta}\leq2\Lambda}(f)
=f\;\;\;\;\mathbb{I}_{\Gamma\leq\sqrt{-\Delta}\leq2\Gamma}(g)=g.
\end{equation*}
Then for any $s_0>\frac{2}{3}$, there exists a $C>0$ such that
\begin{equation}\label{E:bilinearStritwomine}
\|(\nabla_x(e^{it\Delta}f))e^{it\Delta}g\|_{L^2([0,1]\times
M)}\leq C\Lambda({\rm
min}(\Lambda,\Gamma))^{s_0}\|f\|_{L^2(M)}\|g\|_{L^2(M)}
\end{equation}
\end{thm}

After we establish \eqref{E:bilinearStrionemine} and \eqref{E:bilinearStritwomine} to solutions
of~\eqref{E:Schrodinger} satisfying either Dirichlet or Neumann boundary conditions for the general 2
dimensional compact manifolds with boundary,  we will follow Anton's~\cite{Ant06} argument
to prove local well-posedness property in our setting.   \par

We consider the following Cauchy problem in  2-dimensional
compact manifolds with boundary:
\begin{equation}\label{E:Cauchy}
\left\{
\begin{array}{rll}
i\partial_tu+\triangle u &=& \alpha|u|^2u,\;{\rm in}\;\;\mathbb{R}\times M  \\
u|_{t=0}&=& u_0, \;{\rm on}\;\;M \\
u|_{\partial M}&=&0\;({\rm Dirichlet}),\;\;\;(\rm
or)\;\;\;\;N_x\cdot \nabla u|_{\partial M}=0\;\;({\rm Neumann})
\end{array}
\right.
\end{equation}
where $\alpha=\pm 1$. When $\alpha=1$, the equation is defocusing. When $\alpha=-1$, the equation is focusing.
 We consider the local well-posedness property of~\eqref{E:Cauchy}.\par

\begin{defn}
Let $s$ be a real number. We shall say that the Cauchy problem~\eqref{E:Cauchy} is
uniformly well-posed in $H^s(M)$ if, for any bounded subset $B$ of $H^s(M)$,
there exists $T>0$ such that the flow map
\begin{equation*}
u_0\in C^{\infty}(M)\cap B\mapsto u\in C([-T,T],H^s(M))
\end{equation*}
is uniformly continuous when the source space is
endowed with $H^s$ norm, and when the target space is endowed with
\begin{equation*}
\|u\|_{C_TH^s}=sup_{|t|\leq T}\|u(t)\|_{H^s(M)}
\end{equation*}
\end{defn}

Our discussions in the following focus again in 2 dimensional case.
For manifolds without boundary, we only consider first two equations of ~\eqref{E:Cauchy}.
The first result was due to Bourgain~\cite{Bou99} who built the local well-posedness
result in $H^s$ for $s>0$ on the flat torus.
Recently, Burq, G\'{e}rard and Tzvetkov ~\cite{BGT04} use Strichartz estimates to establish
local well-posedness of cubic nonlinear Schr\"{o}dinger equation in $H^s(M)$ for
$s>\frac{1}{2}$ on 2 dimensional manifold without boundary.
In ~\cite{BGT05} they proved the local well-posed property in $H^{s}(M)$ for $s>\frac{1}{4}$
on sphere and Zoll surface by using the bilinear Strichartz estimates~\eqref{E:bilinearStrione}
with $s_0>\frac{1}{4}$.

For manifolds with boundary, it is natural to except a more loss of
derivatives due to Rayleigh whispering galley modes.
In the case of domains of $\mathbb{R}^2$ the local well-posedness
for ~\eqref{E:Cauchy} with Dirichlet boundary condition and
 $s=1$ were proved by Anton~\cite{Ant05}.
On the other direction,
Burq, G\'{e}rard and Tzvetkov ~\cite{BGT03} built an illposedness
result on a disc of $\mathbb{R}^2$, for $s<\frac{1}{3}$.

Our result is the following.\par

\begin{thm}\label{T:well-posed}
If $(M,g)$ is a 2 dimensional manifold with boundary, then the Cauchy problem ~\eqref{E:Cauchy}
is uniformly well-posed
in $H^s(M)$ for every $s>\frac{2}{3}$.
\end{thm}

\section{Reductions}

We start with the proof of Theorem~\ref{T:bilinearone}. The Laplace-Beltrami
operators on $M$ will take the following form in local coordinates
\begin{equation}\label{E:locallaplacian}
(Pf)(x)=\rho^{-1}\sum_{i,j=1}^{n}\partial_i(\rho(x)g^{ij}(x)\partial_jf(x))
\end{equation}

Assume $\mathbb{I}_{\Lambda\leq\sqrt{-\Delta}\leq2\Lambda}(f)=f\;\;,\;\;\mathbb{I}_{\Gamma\leq\sqrt{-\Delta}\leq2\Gamma}(g)=g$ and $\Lambda<\Gamma$.  Then
\begin{equation}\label{E:localgoalone}
\begin{aligned}
\|e^{it\Delta}f\;e^{it\Delta}g\|_{L^2([0,1]\times M)}&\lesssim\|v\|_{L^{\infty}([0,1];L^2(M))}\|u\|_{L^{2}([0,1];L^{\infty}(M))}\\
&\lesssim\|g\|_{L^2(M))}\|u\|_{L^{2}([0,1];L^{\infty}(M))},
\end{aligned}
\end{equation}
where we have used the conservation of mass for the free Schr\"{o}dinger operator in the last inequality.

We define Sobolev spaces on $M$ using the spectral resolution of $P$,
\begin{equation*}
\|f\|_{H^s(M)}=\|\langle D_p\rangle^s f\|_{L^2(M)}\;\;,\;\;\langle D_p\rangle=(1-P)^{\frac{1}{2}}
\end{equation*}
By elliptic regularity (e.g [~\cite{GT83}, Theorem 8.10]) the space $H^s$ coincide with the Sobolev spaces
defined using local coordinates, provided $0\leq s\leq 2$.

Let $r=\frac{2}{3}+\varepsilon>\frac{2}{3}\;,\;s=r-1$. Then we need to establish
\begin{equation*}
\|u\|_{L^2([0,1];L^{\infty}(M))}\lesssim\|f\|_{H^r(M)}\approx({\Lambda})^{r}\|f\|_{L^2(M)},
\end{equation*}
or equivalently,
\begin{equation*}
\|u\|_{L^2([0,1];L^{\infty}(M))}\lesssim\|\Lambda^s f\|_{H^1(M)}
\end{equation*}

By conservation law of free Schr\"{o}dinger operator which is equivalent to
 \begin{equation}\label{E:goalone}
 \|u\|_{L^2([0,1];L^{\infty}(M))}\lesssim
\|\Lambda^s u\|_{L^{\infty}([0,1];H^1(M))}
\end{equation}
Although $(2,2,\infty)$ is not Schr\"{o}dinger admissible, we
should see that once we localize both time and frequency we can
still get desired type of Strichartz estimates.

We work in boundary normal coordinates
for the Riemannian metric $g_{ij}$ that is dual $g^{ij}$ of
(\ref{E:locallaplacian}). Let $x_2>0$ define the manifold $M$, and $x_1$ is a
coordinate function on $\partial M$ which we choose so that
$\partial_{x_1}$ is of unit length along $\partial M$. In these
coordinates,
\begin{equation*}
g_{22}(x_1,x_2)=1\;\;\;\;g_{11}(x_1,0)=1\;\;\;g_{12}(x_1,x_2)=0
\end{equation*}

We now extend the coefficient $g^{11}$ and $\rho$ in an even
manner across the boundary, so that
\begin{equation*}
g^{11}(x_1,-x_2)=g^{11}(x_1,x_2)\;\;\;\rho(x_1,-x_2)=\rho(x_1,x_2).
\end{equation*}

The extended functions are then piecewise smooth, and of Lipschitz
regularity across $x_2=0$. Because $g$ is diagonal, the operator
$P$ is preserved under the reflection $x_2\rightarrow -x_2$.
Eigenspaces for the extended operator $\tilde{P}$ decompose
into symmetric and antisymmetric functions; these correspond to extensions of eigenfunctions for $P$ satisfying
Dirichlet (resp. Neumann) conditions. These eigenfunctions are of $C^{1,1}$ across the boundary.
The Schr\"{o}dinger flow for $P$ is thus extended to $\tilde P.$

Hence matters reduces to considering the Schr\"{o}dinger evolution
on the manifold without boundary with Lipschitz metrics. And we have to show
\begin{equation*}
\|u\|_{L^2([0,1],L^{\infty}(M))}\lesssim\|\Lambda^s u\|_{L^{\infty}([0,1];H^1(M))}
\end{equation*}

By taking a finite partition
of unity, it suffices to prove that
\begin{equation*}
\|\psi u\|_{L^2([0,1];L^{\infty}(\mathbb{R}^2))}\lesssim
\|\Lambda^s u\|_{L^{\infty}([0,1];H^1(M))}
\end{equation*}
for each smooth  cutoff $\psi$ supported in  a suitably chosen
coordinate charts.
 We will choose coordinate
charts such that the image contains the unit ball, and
\begin{equation*}
\|g^{ij}-\delta_{ij}\|_{Lip(B_1(0))}\leq c_0\;\;,\;\;\|\rho-1\|_{Lip(B_1(0))}\leq c_0
\end{equation*}
 for $c_0$ to be taken suitably
small. We take $\psi$ supported in the unit ball, and assume
$g^{ij}$ and $\rho$ are extended so that the above holds
globally on $\mathbb{R}^2$.\par

We denote $u=u_k$ to address that it's frequency being localized to $\Lambda=2^k$, the
estimates we need is now
\begin{equation*}
\|\psi u_k\|_{L^2([0,1];L^{\infty}(\mathbb{R}^2))}\lesssim \|\Lambda^s u_k\|_{L^{\infty}([0,1];H^1(M))}.
\end{equation*}

Let $\{\beta_j(D)\}_{j\geq 0}$ be a Littlewood-Paley partition of
unity on $\mathbb{R}^n$, and let $v_j=\beta_j(D)(\psi u_k)\;,\;v_j^s=(2^j)^sv_j$
, then we will see that it is equivalent to show that for each $j$,
\begin{equation}\label{E:goaltwo}
\|v_j\|_{L^2_tL^{\infty}_x}\lesssim
\|v_j^s\|_{L^{\infty}_tH^1_x}
+({2^j})^{s-1/3}\|(D_t+P)v_j\|_{L^{\infty}_tL^2_x})
\end{equation}
is true, where the norm is taken over $(t,x)=[0,1]\times\mathbb{R}^2$. Note that for any $\varepsilon>0$
\begin{equation*}
\|\psi u_k\|_{L^2_tL^{\infty}_x}\lesssim\|2^{j\varepsilon} v_j \|_{L^2_tL^{\infty}_xl^j_2}
\lesssim\|2^{j\varepsilon} v_j \|_{l^j_2L^2_tL^{\infty}_x}.
\end{equation*}

Here $\varepsilon$ can be absorbed by $s$ in~\eqref{E:goaltwo}, thus we only have to deal
with $\| v_j \|$  instead of $\|2^{j\varepsilon} v_j \|$ in~\eqref{E:goaltwo}.

On the other hand,
$$\begin{aligned}\|v_j^s\|_{L^{\infty}([0,1];H^1({\mathbb{R}^2}))}&
\lesssim\min\{(2^j)\|v_j^s\|_{L^{\infty}([0,1];L^2({\mathbb{R}^2}))}
,(2^j)^{-1}\|v_j^s\|_{L^{\infty}([0,1];H^2({\mathbb{R}^2}))}\}\\
&\lesssim\min\{(2^j)^{1+s}\|u_k\|_{L^{\infty}([0,1];L^2(M))}
,(2^j)^{-1+s}\|u_k\|_{L^{\infty}([0,1];H^2(M))}\}
\end{aligned}$$
To sum up $\|v_j^s\|_{L^{\infty}_tH^1_x}$ over $j$, we dominate those terms with $j\leq k$ by
the first term inside minimum bracket, dominate those terms with $j\geq k$ by the second term inside
minimum bracket. The series is then bounded by a finite sum plus a geometric series.
So the summation over j of first terms in the right hand side
of~\eqref{E:goaltwo} is bounded by
$$\begin{aligned}
(2^k)^{1+s}\|u_k\|_{L^{\infty}([0,1];L^{2}(M))}+(2^k)^{-1+s}\|u_k\|_{L^{\infty}([0,1];H^{2}(M))}
&\lesssim(2^k)^{s}\|u_k\|_{L^{\infty}([0,1];H^1(M))}\\
&\approx \|\Lambda^s u_k\|_{L^{\infty}([0,1];H^1(M))}\end{aligned}$$

For the second term in the right hand side of~\eqref{E:goaltwo}, we note that for a Lipschitz function $a$,
$[\beta_{j}(D),a]:H^{s-1}\rightarrow H^s\;,\;s=0,1$. Hence  $[P,\beta_{j}(D)\psi]:H^1\rightarrow L^2$,
by Coifman-Meyer commutator theorem (see also Proposition 3.6B of ~\cite{Tay91}).
Therefore we have
\begin{equation}\label{E:remainder1}
\|(D_t+P)v_j\|_{L^{\infty}([0,1];L^2(\mathbb{R}^2))}
\lesssim \|u_k\|_{L^{\infty}([0,1]\;H^1(M))}.
\end{equation}
Furthermore, we claim that the following estimate is also true
\begin{equation}\label{E:remainder2}
\|(D_t+P)v_j\|_{L^{\infty}([0,1];L^2(\mathbb{R}^2))}
\lesssim 2^j\|u_k\|_{L^{\infty}([0,1]\;L^2(M))}
\end{equation}
First, we truncate the coefficients of $P$ to frequencies less than some small constant
times $2^j=\eta$ and denote the new coefficients and operator by $g^{ij}_{\eta}$ and $P_j$ respectively.
Note that the localized coefficients satisfy $|g^{ij}-g^{ij}_{\eta}|\lesssim 2^{-j}$. Thus
\begin{equation}\label{E:approxP}
\norm{(P_j-P)v_j}_{L^{\infty}([0,1];L^2(\mathbb{R}^2))}\lesssim 2^j\norm{v_j}_{L^{\infty}([0,1];L^2(\mathbb{R}^2))}
\lesssim 2^j\norm{u_k}_{L^{\infty}([0,1];L^2(M))}.
\end{equation}
Combine this with
\begin{equation*}
\begin{split}
&\|(D_t+P)v_j\|_{L^{\infty}[0,1];L^2(\mathbb{R}^2)}\\
&\leq\norm{(D_t+P_j)v_j}_{L^{\infty}[0,1];L^2(\mathbb{R}^2)}
+\norm{(P-P_j)v_j}_{L^{\infty}[0,1];L^2(\mathbb{R}^2)},
\end{split}
\end{equation*}
we are reduced to estimate
\begin{equation*}
\norm{(D_t+P_j)v_j}_{L^{\infty}[0,1];L^2(\mathbb{R}^2)}.
\end{equation*}

However
\begin{equation}
\begin{split}
\|(D_t+P_j)v_j\|&_{L^{\infty}([0,1];L^2(\mathbb{R}^2))}
\approx 2^j \|(D_t+P_j)v_j\|_{L^{\infty}([0,1];H^{-1}(\mathbb{R}^2))}\\
&\lesssim 2^j\{\norm{(P_j-P)v_j}_{L^{\infty}([0,1];H^{-1}(\mathbb{R}^{2}))}
+\norm{(D_t+P)v_j}_{L^{\infty}([0,1];H^{-1}(\mathbb{R}^{2}))}\}\\
&\lesssim 2^j\|u_k\|_{L^{\infty}([0,1]\;L^2(M))}.
\end{split}
\end{equation}
The first line is due to the localization of $P_j$ and $v_j$.
Next we note that multiplication by a Lipschitz function $\rho$ is a bounded operator in $H^{-1}$.
Thus we regard $P$ and $P_j$ as in divergent form, we can thus bound the first term of the second line as~\eqref{E:approxP}.
While the second term of the second line is also bounded, thanks again to Coifman-Meyer commutator
theorem.

Combine~\eqref{E:remainder1} and~\eqref{E:remainder2}, we thus have
\begin{equation}\|(D_t+P)v_j\|_{L^{\infty}([0,1];L^2(\mathbb{R}^2))}
\lesssim \min\{2^j\|u_k\|_{L^{\infty}([0,1]\;L^2(M)},\|u_k\|_{L^{\infty}([0,1]\;H^1(M))}\}.
\end{equation}

Now we are ready to handle the second term in in the right hand side of~\eqref{E:goaltwo}. For $j\leq k$, we use
\begin{equation*}
 (2^j)^{s-1/3}\|(D_t+P)v_j\|_{L^{\infty}([0,1]\;L^2(\mathbb{R}^2))}\leq (2^j)^{s-1/3}2^j\|u_k\|_{L^{\infty}([0,1]\;L^2(M))}.
\end{equation*}
Therefore the sum of $j=1,\cdots,k$ terms will be bounded by
\begin{equation}\label{E:secondtermone}
C(2^k)^{1/3+\varepsilon}\|u_k\|_{L^{\infty}([0,1]\;L^2(M))}.
\end{equation}

For $j\geq k$, we use
\begin{equation*}
(2^j)^{s-1/3}\|(D_t+P)v_j\|_{L^{\infty}([0,1];L^2({\mathbb{R}^2}))}
\lesssim (2^{(j-k)})^{s-1/3}(2^{-k})^{1/3}\|\Lambda^s u_k\|_{L^{\infty}([0,1];H^1(M))}
\end{equation*}
Since $s-1/3<0$, the sum of $j\geq k$ terms is bounded by
\begin{equation}\label{E:secondtermtwo}
(2^{-k})^{1/3}\|\Lambda^s u_k\|_{L^{\infty}([0,1]\;H^1(M))}.
\end{equation} Thus the sum of~\eqref{E:secondtermone} and ~\eqref{E:secondtermtwo} is bounded
by $\|\Lambda^s u_k\|_{L^{\infty}([0,1];H^1(M))}$.

Now let $\lambda=2^j\;,\;w_{\lambda}=v_j$, ~\eqref{E:goaltwo} can be written as
\begin{equation*}
\|w_{\lambda}\|_{L^2([0,1];L^{\infty}(\mathbb{R}^2))}\lesssim \lambda^{\frac{2}{3}+\varepsilon}
(\|w_{\lambda}\|_{L^{\infty}([0,1];L^2({\mathbb{R}^2}))}
+{\lambda}^{-\frac{4}{3}}\|(D_t+P)w_\lambda\|_{L^{\infty}([0,1];L^2(\mathbb{R}^2))})
\end{equation*}
which is implied by showing for each interval $I_{\lambda}$ with
length $\lambda^{-\frac{4}{3}}$, we all have
\begin{equation*}
\|w_\lambda\|_{L^2(I_{\lambda};L^{\infty}(\mathbb{R}^2))}\lesssim (\lambda)^{\varepsilon}(\|w_\lambda\|_{L^{\infty}(I_{\lambda}
;L^2({\mathbb{R}^2}))}+\|(D_t+P)w_\lambda\|_{L^1(I_{\lambda};L^2(\mathbb{R}^2))})
\end{equation*}

Recall that the operator $P$ here is rough. Thus
we regularize the coefficients of $P$ by setting
\begin{equation*}
g^{ij}_{\lambda}=S_{\lambda^{2/3}}(g^{ij}),\;\;\;\rho_{\lambda}=S_{\lambda^{2/3}}(\rho)
\end{equation*}
where $S_{\lambda^{2/3}}$
denotes a truncation of a function to frequencies less than
$\lambda^{\frac{2}{3}}$. Let $P_{\lambda}$ be the operator with
coefficients $g_{\lambda}^{ij}$ and $\rho_\lambda$. Then
\begin{equation*}
\|(P-P_{\lambda})w_\lambda\|_{L^1(I_{\lambda};L^2(\mathbb{R}^2))}\lesssim \|w_\lambda\|_{L^{\infty}
(I_{\lambda};L^2(\mathbb{R}^2))}
\end{equation*}
since we know
\begin{equation*}
|g^{ij}_{\lambda}-g^{ij}|\lesssim \lambda^{-\frac{2}{3}}
\end{equation*} and
similarly for $\rho$.

Then we rescale the problem by letting $\mu=\lambda^{\frac{2}{3}}$ and define
\begin{equation*}
u_\mu(t,x)=w_{\lambda}(\lambda^{-\frac{2}{3}}t,\lambda^{-\frac{1}{3}}x),\;
\;\;\;Q_{\mu}=P_{\lambda}(\lambda^{-\frac{1}{3}}x,D)
\end{equation*}
The function $u_{\mu}(t,\dot)$ is localized to frequencies of size $\mu$, and the coefficients of $Q_{\mu}$ are
localized to frequencies of the size less than $\mu^{\frac{1}{2}}$. This implies the following estimates of the
coefficients of $Q_{\mu}$
\begin{equation*}
\|\partial_x^{\alpha}g^{ij}_{\lambda}(\lambda^{-\frac{1}{3}}x)\|
+\|\partial_x^{\alpha}\rho_{\lambda}(\lambda^{-\frac{1}{3}}x)\|\leq
C_\alpha \mu^{\frac{1}{2}{\rm max}(0,|\alpha|-2)}.
\end{equation*}

The time interval $I_\lambda$ scales to $\mu^{-1}.$ Also note that by our reduction
$\|g_{\lambda}^{ij}-\delta^{ij}\|_{C^2}\ll 1$. Thus we have reduced the proof of Theorem~\ref{T:bilinearone} to
the following

\begin{thm}\label{T:localStri}
Suppose that $u(t,x)$ is localized to frequencies
$|\xi|\in[\frac{1}{4}\lambda,4\lambda]$ and solves

\begin{equation*}
 (D_t+\sum_{1\leq i,j\leq n}a^{ij}(x)\partial_{x_i}\partial_{x_j}+\sum_{1\leq i\leq n} b^i(x)\partial_{x_i})u=F
\end{equation*}
Assume also that the metric satisfies
\begin{equation*}
\|a^{ij}-\delta_{ij}\|_{C^2}\ll 1,\;\;\;\;\;\|b^i\|_{C^1}\lesssim 1
\end{equation*}
\begin{equation*}
{\rm supp}(\widehat{a^{ij}}),\;{\rm supp}(\widehat{b^{i}})\subset B_{\lambda^{1/2}}(0).
\end{equation*}
Then the following estimate holds
\begin{equation*}
\|u\|_{L^2([0,\lambda^{-1}];L^{\infty}(\mathbb{R}^2))}\lesssim (\log \lambda)^{\frac{1}{2}}(\|u\|_{L^{\infty}([0,\lambda^{-1}]
;L^2({\mathbb{R}^2}))}+\|F\|_{L^1([0,\lambda^{-1}];L^2(\mathbb{R}^2))})
\end{equation*}
\end{thm}

\section{Wave Packet and Parametrix}

To prove Theorem~\ref{T:localStri}, we need some notations for wave packet transform.
We fix a real, radial Schwartz function $g(x)\in\mathcal{S}(\mathbb{R}^2)$, with
$\|g\|_{L^2}=(2\pi)^{-1}$, and assume its Fourier transform
$h(\xi)=\hat{g}(\xi)$ is supported in the unit ball $\{|\xi|<1\}$.
For $\lambda\geq 1$, we define
$T_{\lambda}:\mathcal{S}'(\mathbb{R}^2)\rightarrow\mathcal{C}^{\infty}(\mathbb{R}^{4})$by
\begin{equation*}
(T_{\lambda}f)(x,\xi)=\lambda^{\frac{1}{2}}\int e^{-i\lr{\xi, z-x}}g(\lambda^{\frac{1}{2}}(z-x))f(z)dz.
\end{equation*}
A simple calculation shows that
\begin{equation*}
f(y)={\lambda}^{\frac{1}{2}}\int e^{i\lr{\xi,y-x}}g({\lambda}^{\frac{1}{2}}(y-x))(T_{\lambda}f)(x,\xi)dxd\xi,
\end{equation*}
so that $T_{\lambda}^{\ast}T_{\lambda}=I.$ In particular,
\begin{equation*}
\|T_{\lambda}f\|_{L^2(\mathbb{R}^{4}_{x,\xi})}=\|f\|_{L^2(\mathbb{R}^2_x)}.
\end{equation*}

Let
\begin{equation*}
D_t+A(x,D)+B(x,D)=D_t+\sum_{1\leq i,j\leq
n}a^{ij}(x)\partial_{x_i}\partial_{x_j}+
\sum_{1\leq i\leq n}b^i\partial_{x_i}.
\end{equation*}
We conjugate $A(x,D)$ by $T_{\lambda}$ and take a suitable
approximation to the resulting operator. Define the following
differential operator over $(x,\xi)$
\begin{equation*}
\widetilde{A}=-id_{\xi}a(x,\xi)\cdot d_x+id_x a(x,\xi)\cdot d_{\xi}+a(x,\xi)-\xi\cdot d_{\xi}a(x,\xi)
\end{equation*}
By the argument from wave packet methods (Lemmas 3.1-3.3 in
Smith~\cite{Smi06}), we have that if $\widetilde{\beta}_{\lambda}$
is a Littlewood-Paley cutoff truncating to frequencies
$|\xi|\approx \lambda$ then
\begin{equation*}
\|T_{\lambda}A(\cdot,D)\widetilde{\beta}_{\lambda}(D)-\widetilde{A}T_{\lambda}\widetilde{\beta}_{\lambda}(D)\|
_{L^2_{x}\rightarrow L^2_{x,\xi}}\lesssim \lambda
\end{equation*}

This yields that, if
$\tilde{u}(t,x,\xi)=(T_{\lambda}u(t,\cdot))(x,\xi)$, then
$\tilde{u}$ solves the equation
\begin{equation*}
(\partial_{t}+d_{\xi}a(x,\xi)\cdot d_x-d_x a(x,\xi)\cdot d_{\xi}+ia(x,\xi)-i\xi\cdot d_{\xi}a(x,\xi))
\tilde{u}(t,x,\xi)=\tilde{G}(t,x,\xi)
\end{equation*}
where $\tilde{G}$ satisfies
\begin{equation*}\int_{0}^{\lambda^{-1}}\|\tilde{G}(t,x,\xi)\|_{L^2_{x,\xi}}dt\lesssim \|u\|_{L^{\infty}([0,\lambda^{-1}];L^2])}
+\|F\|_{L^1([0,\lambda^{-1}];L^2)}
\end{equation*}

Given an integral curve $\gamma(r)\in\mathbb{R}^{4}_{x,\xi}$ of
the vector field
\begin{equation*}
\partial_t+d_{\xi}a(x,\xi)\cdot d_x-d_xa(x,\xi)\cdot d_{\xi}
\end{equation*}
with $\gamma(t)=(x,\xi)$, we denote
$\chi_{s,t}(x,\xi)=(x_{s,t},\xi_{s,t})=\gamma(s)$. Also define
\begin{equation*}
\sigma(x,\xi)=a(x,\xi)-\xi\cdot d_{\xi}a(x,\xi),\;\;\;\; \psi(t,x,\xi)=\int_{0}^t\sigma(\chi_{r,t}(x,\xi))dr
\end{equation*}

This allows us to write
\begin{equation*}
{\tilde u}(t,x,\xi)=e^{-i\psi(t,x,\xi)}{\tilde u}_0(\chi_{0,t}(x,\xi))+\int_0^te^{-i\psi(t-r,x,\xi)}\tilde G(r,
\chi_{r,t}(x,\xi))dr
\end{equation*}
where $\tilde u$ is an integrable superposition over $r$ of functions invariant under the
flow of $\tilde A$, truncated to $t>r$.

Since $u(t,x)=T_{\lambda}^{\ast}\tilde u(t,x,\xi)$ it thus
suffices to obtain estimates
\begin{equation}\label{E:wavepacketesta}
\|{\tilde
\beta_{\lambda}}(D)W_tf\|_{L^2_tL^{\infty}_{x}}\lesssim (\log
\lambda)^{\frac{1}{2}}\|f\|_{L^2_{x,\xi}}
\end{equation} where
$W_t$ acts on function $f(x,\xi)$ by the formula
\begin{equation}\label{E:wavepacketformula}
(W_tf)(y)=T^{\ast}_{\lambda}(e^{-i\psi(t,x,\xi)}f(\chi_{0,t}(\cdot)))(y)
\end{equation}

In order to get the desired estimates by $TT^{\ast}$ method, we
investigate the kernel $K(t,y,s,x)$ of $W_tW_s^{\ast}$ which is
\begin{equation*}
\lambda\int e^{-i\lr{\zeta,x-z}-i\int_{s}^{t}\sigma(\chi_{r,t}(z,\zeta))+i\lr{\zeta_t,y-z_t}}
g(\lambda^{\frac{1}{2}}(y-z_{t,s}))g(\lambda^{\frac{1}{2}}(x-z))dzd\zeta
\end{equation*}
Recall that ${\rm supp}(\hat g)\subset B_1(0)$. We are concerned
with
$\tilde{\beta}_{\lambda}W_tW^{\ast}_s\tilde{\beta}_{\lambda}$,
thus we can inserted a cutoff $S_{\lambda}(\zeta)$  into the
integrand which is supported in a set $|\zeta|\approx \lambda$.
Also note that the Hamiltonian vector field is independent of
time, that is $\chi_{t,s}=\chi_{t-s,0}$. We denote it by
$\chi_{t-s,0}(z,\zeta)=\chi_{t-s}(z,\zeta)=
(z_{t-s},\zeta_{t-s})$. It then suffices to consider $s=0$, and
the kernel $K(t,x,0,y)$ as
\begin{equation*}
\lambda\int e^{-i\lr{\zeta,x-z}-i\psi(t,z,\zeta)+i\lr{\zeta_t,y-z_t}}g(\lambda^{\frac{1}{2}}(y-z_{t,s}))
g(\lambda^{\frac{1}{2}}(x-z))S_{\lambda}(\zeta)dzd\zeta
\end{equation*}

We will built the estimates~\eqref{E:wavepacketesta} by considering the estimate for
time variable between $[0,\lambda^{-2}]$ and
$[\lambda^{-2},\lambda^{-1}]$ respectively. That is we will prove
\begin{equation}\label{E:waveesta}
\|{\tilde
\beta_{\lambda}}(D)W_tf\|_{L^2([0\;,\;\lambda^{-2}];L^{\infty}({\mathbb{R}^2}))}\lesssim\|f\|_{L^2_{x,\xi}}
\end{equation}
and
\begin{equation}\label{E:waveestb}
\|{\tilde\beta_{\lambda}}(D)W_tf\|_{L^2([\lambda^{-2}\;,\;\lambda^{-1}];L^{\infty}(\mathbb{R}^2))}\lesssim
(\log \lambda)^{\frac{1}{2}}\|f\|_{L^2_{x,\xi}}
\end{equation}
The inequality~\eqref{E:waveesta} is easy to prove , note that when
$t\in[0,\lambda^{-2}]$, it is easy to see that
\begin{equation}
|K(t,x,0,y)|\approx \lambda\cdot(\lambda^{-\frac{1}{2}})^2\cdot\lambda^2=\lambda^2.
\end{equation}

The term $(\lambda^{-\frac{1}{2}})^2$ came from the size of $g$ and $\lambda^2$ from $S_\lambda$.
Then the estimates follows from applying Schwartz inequality to time variables.

The inequality ~\eqref{E:waveestb} comes from establishing
\begin{equation}\label{E:dispersion}
|K(t,x,0,y)|\lesssim \frac{1}{t}
\end{equation}
 for $t\in[\lambda^{-2},\varepsilon\lambda^{-1}]$ with $\varepsilon$ chosen sufficient small and independent of
 $\lambda$. Then by Schwartz inequality, we get
\begin{equation*}
\|\tilde{\beta}_\lambda W_tW_s^\ast\tilde{\beta}_\lambda\|_{L^2\rightarrow L^2}\lesssim
\int_{\lambda^{-2}}^{\lambda^{-1}}\frac{1}{t}dt=\log\lambda.
\end{equation*}

The dispersive estimate~\eqref{E:dispersion} we need is actually proved
in the section 4 of Blair, Smith and Sogge~\cite{BSS07}. Hence we conclude Theorem~\ref{T:localStri}.

\section{Gradient Estimates}

Next we will prove Theorem~\ref{T:bilineartwo}.
Recall that we assume
\begin{equation*}
\mathbb{I}_{\Lambda\leq\sqrt{-\Delta}\leq2\Lambda}(f)
=f\;\;,\;\;\mathbb{I}_{\Gamma\leq\sqrt{-\Delta}\leq2\Gamma}(g)=g.
\end{equation*}
If $\Lambda>\Gamma$, we can prove as following
$$\begin{aligned}
\|\nabla(e^{it\triangle}f)e^{it\triangle}g\|_{L^2([0,1]\times M)}&\lesssim
\|\nabla e^{it\triangle}f\|_{L^{\infty}([0,1];L^2(M))}\|e^{it\triangle}g\|_{L^2([0,1]L^{\infty}(M))}\\
&\lesssim\Lambda\|e^{it\triangle}f\|_{L^{\infty}([0,1];L^2(M))}\Gamma^s\|g\|_{L^2(M)} \\
&\lesssim\Lambda\Gamma^s\|f\|_{L^2(M)}\|g\|_{L^2(M)}, \end{aligned}$$ where we have used the fact Riesz transform $\nabla({-\triangle})^{-1/2}$ is bounded on $L^2(M)$
(see~\cite{Shen05}) and then apply H\"{o}rmander multiple theorem (see~\cite{Xu04}) in the second inequality.\par

If $\Lambda<\Gamma$, as the reduction~\eqref{E:localgoalone}, Let
$r=\frac{5}{3}+\varepsilon\;,\;s=r-1$. Then we need
to prove that
$$\|\nabla u\|_{L^2([0,1];L^{\infty}(M))}\lesssim \|\Lambda^{s}f\|_{H^1(M)}$$
is true. Again we write it as
\begin{equation}\label{E:derivativeest}
\|\nabla u_k\|_{L^2([0,1];L^{\infty}(M))}\lesssim \|\Lambda^{s}u_k\|_{H^1(M)}
\end{equation}
 for
denoting that it's frequency being localized to $\Lambda=2^k$. By making use of
the following inequality
\begin{equation}\label{E:gradientest}
\|\nabla u_k\|_{L^2([0,1];L^{\infty}(M))}\lesssim \Lambda
\| u_k\|_{L^2([0,1];L^{\infty}(M))}
\end{equation}
and estimate~\eqref{E:goalone} we conclude the result.

To see~\eqref{E:gradientest} is true, we will use an argument concerning finite speed of
propagation of wave equation
(see for example~\cite{Sog02},~\cite{Xu04} ) and  the following gradient estimate of unit band
spectral projection operator.  The unit band spectral projection operator is defined as
$$\chi_{\lambda}f(x)=\sum_{\lambda\leq\lambda_k<\lambda+1}E_kf(x)
=\sum_{\lambda\leq\lambda_k<\lambda+1}e_k(x)\int_{M}f(y)e_k(y)dy$$

\begin{thm}[~\cite{Xu04}~Theorem 1]
Fix a compact Riemannian manifold $(M,g)$ with boundary and ${\rm dim}M=n$, for both Dirichlet Laplacian and
Neumann Laplacian
on $M$, there is a uniform constant $C$ such that
\begin{equation}\label{E:gradientofXu}
\|\nabla\chi_{\lambda}f\|_{L^{\infty}(M)}\leq C \lambda^{(n+1)/2}\|f\|_{L^2(M)}
\end{equation}
\end{thm}
In fact, we are going to use it's dual form , that is
\begin{equation}\label{E:gradientofXudual}
\|\chi_{\l}\nabla f\|_{L^{2}(M)}\leq C \l^{(n+1)/2}\|f\|_{L^1(M)}
\end{equation}

Let $\{\beta_j\}_{j\geq 0}$ be a Littlewood-Paley partition on $\mathbb{R}$.
Since Littlewood -Paley
operator commutes with Schrodinger operator, estimate~\eqref{E:gradientest}
will be a consequence of
\begin{equation}\label{E:gradientest2}
\|\nabla \beta_k(D) f\|_{L^{\infty}(M)}\lesssim \lambda\|f\|_{L^{\infty}(M)}
\end{equation}
where $2^k=\lambda$ and $f$ is spectrally localized to on dyadic interval of order $\lambda$.
However we should prove the following dual inequality
\begin{equation}\label{E:gradientest3}
\| \beta_k(D)\nabla f\|_{L^{1}(M)}\lesssim \l\|f\|_{L^{1}(M)},
\end{equation}
 since this implies~\eqref{E:gradientest2}.\par

Recall that $\beta_j(\cdot)=\beta(\frac{\cdot}{2^j})\;,\;j\geq 1$ for some $\beta\in C^{\infty}_{0}(1/2,4)$.
We may assume it is an even function on $\mathbb{R}$, otherwise we only need replace $\beta(t)$ by
$\tilde{\beta(t)}$ where the even function $\tilde{\beta(t)}=\beta(t)$ for $t>0$.
Write
$$\beta(\frac{P}{\lambda})\nabla f(x)
=\frac{1}{2\pi}\int_{\mathbb{R}}\lambda\widehat{\beta}(\lambda t) e^{itP}\nabla f(x)dt.$$
Note that proving~\eqref{E:gradientest3} is equivalent to considering
$$T_{\lambda}(P)f(x)=\int_{\mathbb{R}}\lambda\widehat{\beta}(\lambda t)\cos{tP}\nabla f(x)dt,$$
and proving
\begin{equation}\label{E:estimateofT}
\|T_{\lambda}(P)f\|_{L^{1}(M)}\lesssim\lambda\|f\|_{L^{1}(M)}
\end{equation}
Here $P=\sqrt{-\triangle}$ and
$$\cos tP\nabla f(x)=\sum_{k=1}^{\infty} \cos t\lambda_kE_k(\nabla f)(x)=u(t,x)$$
is the cosine transform of $\nabla f$.
It is the solution of wave equation
$$(\partial_t^2-\triangle_g)u=0\;,\;u(0,\cdot)=\nabla f\;,\;u_t(0,\cdot)=0.$$

In order to prove~\eqref{E:estimateofT} , we shall use the finite propagation speed for
solutions to the wave equation. Specifically, if
$\nabla f$ is supported in a geodesic ball $B(x_0,R)$ centered at $x_0$ with radius $R$, then $x\To \cos tP\nabla f$
vanishes outside of $B(x_0, 2R)$ if $0\leq t\leq R$.\par

Let $1=\eta(t)+\sum_{j=1}^{\infty}\rho(2^{-j}t)$ be a Littlewood-Paley partition of $\mathbb{R}$. Write
$T_{\l}=T_{\l}^0+T_{\l}^j$, here
\begin{equation}\label{E:Tzero}
T_{\lambda}^0(P) f=\int_{\mathbb{R}}\eta(\lambda t)\lambda\widehat{\beta}(\lambda t)\cos{tP}\nabla fdt
\end{equation} and
\begin{equation}\label{E:Tj}T_{\lambda}^j(P) f=\int_{\mathbb{R}}\rho(2^{-j}\lambda t)
\lambda\widehat{\beta}(\lambda t)\cos{tP}\nabla fdt
\end{equation}
We will prove $T_{\lambda}(P)$ satisfies~\eqref{E:estimateofT} by showing $T_{\lambda}^0(P)$ and
$\sum_{j\geq1}T_{\lambda}^j(P)$ both satisfy~\eqref{E:estimateofT}.\par

Now
$$
\begin{aligned}
T_{\lambda}^0(P) f(x)&=\int_{\mathbb{R}}\eta(\lambda t)\lambda\widehat{\beta}(\lambda t)\cos{tP}\nabla f(x)dt\\
&=\int_{\mathbb{R}}\eta(\lambda t)\lambda\widehat{\beta}(\lambda t)\sum_{\lambda\leq \lambda_k\leq 2\lambda}
 \cos t\lambda_k e_{_k}(x)\int_M e_{k}(y)\nabla f(y)dydt\\
&=\int_M\{\int_{\mathbb{R}}\eta(\lambda t)\lambda\widehat{\beta}(\lambda t)\sum_{\lambda\leq \lambda_k\leq 2\lambda}
 \cos t\lambda_k e_{k}(x) e_{k}(y)dt\}\nabla f(y)dy\\
&=\int_M K_{\lambda}^0(x,y)f(y)dy
\end{aligned}
$$
Because the finite propagation speed of the wave equation mentioned before implies that
the kernel of the operator $K_{\lambda}^0(x,y)$ must satisfy
$$K_{\lambda}^0(x,y)=0 \;\;\;\;{\rm if}\;\;\;\; {\rm dist}(x,y)>8\lambda^{-1},$$
since $\cos tP$ will have a kernel that vanishes on this set when $t$ belongs to the support of
the integral defining $K_{\lambda}^0(x,y)$. Because of this, in order to prove  $T_{\l}^0$ satisfies~\eqref{E:estimateofT}, it suffices to show that for all geodesic balls $B_{\l,0}$ with radius
$8\l^{-1}$ one has the bound
\begin{equation}\label{E:estimateofTzero}
\norm{T_{\l}^0 f}_{L^1(B_{\l,0})}\lesssim\lambda\|f\|_{L^{1}(M)},
\end{equation}
For the $L^1$ norm over $B_{\l,0}$.  Also we want to use~\eqref{E:gradientofXu}, so rewrite
$$\nabla f=\sum_{l}{\nabla f}_l=\sum_{l=\lambda}^{2\lambda-1}\chi_{l}\nabla f$$
with each $\nabla f_l$ being spectrally localized to unit band.\par

By using Cauchy-Schwartz inequality,~\eqref{E:gradientofXudual}, and orthogonality we find
\begin{equation}\label{E:gradientestmain}
\begin{split}
\norm{T_{\l}^0f}_{L^1(B_{\l,0})}&\leq 8{\l}^{-1}\norm{T_{\l}^0f}_{L^2(M)}\\
&\leq C{\l}^{-1}\{\sum_{l=\l}^{2\l}
(\sup_{l\leq\l_k<l}|\beta(\frac{\l_k}{\l})|^2)\norm{\chi_{l}\nabla f}^2_{L^2(M)}\}^{1/2}\\
&\leq C{\l}^{-1}{\l}^{1/2}{\l}^{3/2}\norm{f}_{L^1(M)}.
\end{split}
\end{equation}

Similar,
\begin{equation}\label{E:Linfiniteboundthree}
\begin{aligned}
T_{\lambda}^j f(x)&=\int_{\mathbb{R}}\rho(2^{-j}\lambda t)\lambda\widehat{\beta}(\lambda t)\cos{tP}\nabla f(x)dt\\
&=\int_M\{\int_{\mathbb{R}}\rho(2^{-j}\lambda t)\lambda\widehat{\beta}(\lambda t)
\sum_{\lambda\leq \lambda_k\leq 2\lambda} \cos t\lambda_k e_{k}(x) e_{k}(y)dt\}\nabla f(y)dy\\
&=\int_M K_{\lambda}^j(x,y)f(y)dy
\end{aligned}
\end{equation}
has the property that $K_{\lambda}^j(x,y)=0$ if
${\rm dist(x,y)}\geq 8\cdot 2^{j+1}\cdot\lambda^{-1}$.
Note that the dyadic cutoff localizes to $|t|\approx \lambda^{-1}2^j$. Hence
follows again~\eqref{E:gradientestmain} yields the bound
${2^{j+1}}{\l}^{-1}\l^{1/2}(\l t)^{-N}(\lambda^{3/2})\|f\|_{1}$ with $N$ be
a large enough positive integer. Here the term $2^{j+1}\l^{-1}$ comes
from the volume of geodesic ball $B_{\l,j}$ with radius $8\cdot 2^{j+1}\cdot\lambda^{-1}$
 , $(\lambda t)^{-N}$ from value of $\beta$.
 Thus we have
$$\|T_{\lambda}^j \|_{L^{1}(B_{\l,j})}\lesssim \lambda 2^{-jN}\|f\|_{L^{1}(M)}$$
which form a
geometric series and thus the sum of $j=1,\cdots,\infty$ terms enjoys the
property~\eqref{E:estimateofT}. \par

\section{Cubic NLS}
\subsection{Cauchy Problem}
In the following, we establish the well-posedness of the cubic
nonlinear Schr\"{o}dinger equation in 2 dimensional compact manifolds
$(M,g)$ with boundary. The equations we are interested in is following.
\begin{equation}\label{E:Cauchyproblem}
\left\{
\begin{array}{rll}
i\partial_tu+\triangle u &=& \alpha|u|^2u,\;{\rm on}\;\;\mathbb{R}\times M  \\
u|_{t=0}&=& u_0, \;{\rm on}\;\;M \\
u|_{\partial M}&=&0\;({\rm Dirichlet}),\;\;\;(\rm
or)\;\;\;\;N_x\cdot \nabla u|_{\partial M}=0\;\;({\rm Neumann})
\end{array}
\right.
\end{equation}
where $\alpha=\pm 1$.

\begin{defn}
Let $s$ be a real number. We shall say that the Cauchy problem~\eqref{E:Cauchyproblem}
is uniformly well-posed in $H^s(M)$ if,
for any bounded subset of $H^s(M)$, there exists $T>0$ such that the flow map
$$u_0\in C^{\infty}(M)\cap B\mapsto u\in C([-T,T],H^s(M))$$
is uniformly continuous when the source space is
endowed with $H^s$ norm, and when the target space is endowed with
$$\|u\|_{C_TH^s}=sup_{|t|\leq T}\|u(t)\|_{H^s(M)}$$
\end{defn}

Let's state again our local well-posdness results Theorem~\ref{T:well-posed}.
\begin{thma}
If $(M,g)$ is a 2 dimensional manifold with boundary, then the Cauchy problem for~\eqref{E:Cauchyproblem}
is uniformly well-posed in $H^s(M)$ for every $s>\frac{2}{3}$.
\end{thma}

\subsection{Bourgain Spaces}
In order to prove the local well-posedness of cubic nonlinear
Schr\"{o}dinger equation on manifolds with boundary. We introduce
Bourgain space $X^{s,b}$. Our definition follows from Burq,
G\'{e}rard and Tzvetkov ~\cite{BGT05} using the spectral projectors
on manifolds.

Let $(e_k)$ be a $L^2(M)$ orthonormal basis of eigenfunctions of
Dirichlet(or Neumann) Laplacian $-\triangle_g$ with eigenvalues
$\mu_k^2$, $E_k$ be the orthogonal projector along $e_k$. The
Sobolev space $H^s(M)$ is associated to $(I-\triangle)^{1/2}$,
equipped with the norm
$$\|u\|^2_{H^s(M)}=\sum_{k}\langle\mu_k\rangle^{2s}\|E_ku\|^2_{L^2(M)}$$
where $\langle\mu_k\rangle=(1+\mu_k^2)^{\frac{1}{2}}$.

\begin{defn}
The space $X^{s,b}(\mathbb{R}\times M)$ is the completion of
$C^{\infty}_{0}(\mathbb{R}_t;H^s(M))$ with the norm
\begin{align}\label{E:definitionBourgain}
\|u\|^2_{X^{s,b}(\mathbb{R}\times
M)}&=\sum_k\|\langle\tau+\mu_k^2\rangle^b\langle\mu_k\rangle^{s}
\widehat{E_ku}(\tau)\|^2_{L^2(\mathbb{R}_\tau;L^2(M))} \\
&=\|e^{-it\triangle}u(t,\cdot)\|^2_{H^b(R_t;H^s(M))}
\end{align} where $\widehat{E_ku}(\tau)$ denote the Fourier
transform of $E_ku$ with respect to the time variable.
\end{defn}

In fact, if $s\geq 0$ and $u\in \mathcal{S}'(\mathbb{R},L^2(M))$.
Let $F(t,\cdot)=e^{-it\triangle}u(t,\cdot)$, then $F(t,\cdot)\in
\mathcal{S}'(\mathbb{R},L^2(M))$ and
$E_k(F(t,\cdot))=e^{it\mu_k^2}E_k(u(t,\cdot))$. Hence
$\widehat{E_k(F)}(\tau)=\widehat{E_k(u)}(\tau-\mu_k^2)$. Applies this to~\eqref{E:definitionBourgain} ,
 we conclude
$$\|u\|^2_{X^{s,b}(\mathbb{R}\times M)}=\|e^{-it\triangle}u(t,\cdot)\|^2_{H^b(R_t;H^s(M))}.$$

We also note that if $b>\frac{1}{2}\;,\;H^{b}(\mathbb{R},H^s(M))\hookrightarrow C(\mathbb{R},H^s(M))$,
since $u(t,\cdot)=e^{it\triangle}F(t,\cdot)$, we have $u\in C(\mathbb{R}, H^s(M))$.

In order to use a contraction mapping argument to obtain local
existence. We need to define local in time version of
$X^{s,b}(\mathbb{R}\times M)$. For $T>0$ we denoted by
$X^{s,b}_T(M)$ the space of restrictions of elements of
$X^{s,b}(\mathbb{R}\times M)$ endowed with the norm
$$\|u\|_{X^{s,b}_T}=\inf\{\|\tilde{u}\|_{X^{s,b}(\mathbb{R}\times M)}\;,\;\tilde u|_{(-T,T)\times M}=u\}$$

Now we can reformulate the bilinear estimates in the $X^{s,b}$
content. The following lemma should refer to the lemma 2.3 of
~\cite{BGT05}.

\begin{lem}
Let $s\in\mathbb{R}$. The following statements are equivalent:

\noindent(1) For any $u_0\;,\;v_0\in L^2(M)$ satisfying
$$1_{\lambda\leq\sqrt{-\triangle}\leq2\lambda}u_0=u_0\;\;\;,\;\;\;1_{\mu\leq\sqrt{-\triangle}\leq2\mu}v_0=v_0$$
one has
\begin{equation}\label{E:bilinearStrioneone}
\|e^{it\triangle}u_0\;e^{it\triangle}v_0\|_{L^2((0,1)_t\times
M)}\leq C({\rm min}(\lambda,\mu))^s
\|u_0\|_{L^2(M)}\|v_0\|_{L^2(M)}
\end{equation}
(2)For any $b>\frac{1}{2}$ and any $f,g\in
X^{0,b}(\mathbb{R}\times M)$ satisfying
$$1_{\lambda\leq\sqrt{-\triangle}\leq2\lambda}f=f\;\;\;,\;\;\;1_{\mu\leq\sqrt{-\triangle}\leq2\mu}g=g$$
one has
\begin{equation}\label{E:bilinearBourgain}
\|fg\|_{L^2(\mathbb{R}\times M)}\leq
C({\rm min}(\lambda,\mu))^s\|f\|_{X^{0,b}(\mathbb{R}\times M)}
\|g\|_{X^{0,b}(\mathbb{R}\times M)}
\end{equation}
\end{lem}
\begin{proof}
If $u(t)=e^{-it\triangle}u_0$ then for any $\psi\in C^{\infty}_{0}(\mathbb{R})$ and any $b\;,\;\psi(t)u(t)\in
X^{0,b}(\mathbb{R}_t\times M)$ with $$\|\psi u\|_{X^{0,b}(\mathbb{R}\times M)}\leq C\|u_0\|_{L^2(M)}$$ which
shows that~\eqref{E:bilinearBourgain} implies~\eqref{E:bilinearStrioneone}.

Suppose that $f(t)$ and $g(t)$ are supported in time in the interval $(0,1)$ and write
$$f(t)=e^{it\triangle}e^{-it\triangle}f(t)=e^{it\triangle}F(t)\;,\;g(t)=e^{it\triangle}e^{-it\triangle}g(t)
=e^{it\triangle}G(t)$$
Then
$$f(t)=\frac{1}{2\pi}\int_{-\infty}^{\infty}e^{it\tau}e^{it\triangle}\widehat{F}
(\tau)d\tau\;,\;g(t)=\frac{1}{2\pi}\int_{-\infty}^{\infty}e^{it\tau}e^{it\triangle}\widehat{G}(\tau)d\tau$$
and hence
$$(fg)(t)=\frac{1}{(2\pi)^2}\int_{-\infty}^{\infty}\int_{-\infty}^{\infty}e^{it(\tau+\sigma)}e^{it\triangle}
\widehat{F}(\tau)e^{it\triangle}\widehat{G}(\sigma)d\tau d\sigma.$$

Ignoring the oscillating factors
$e^{it(\tau+\sigma)}$, using~\eqref{E:bilinearStrioneone} and the Cauchy-Schwartz inequality
in $(\tau,\sigma)$ (in this places we use
that $b>\frac{1}{2}$ to get the needed integrability) yields
\begin{align}
\|fg\|_{L^2((0,1)\times M)}&\leq C({\rm min}(\lambda,\mu))^s\int_{\tau,\sigma}\|\widehat{F}(\tau)\|_{L^2(M)}
\|\widehat{G}(\sigma)\|_{L^2(M)}d\tau d\sigma \notag\\
&\leq C({\rm min}(\lambda,\mu))^s\|\langle\tau\rangle^b\widehat{F}(\tau)\|_{L^2(\mathbb{R}_\tau\times M)}
\|\langle\sigma\rangle^b\widehat{G}(\sigma)\|_{L^2(\mathbb{R}_\sigma\times M)}\\
&=C({\rm min}(\lambda,\mu))^s\|f\|_{X^{0,b}(\mathbb{R}\times
M)}\|g\|_{X^{0,b}(\mathbb{R}\times M)} \notag
\end{align} Finally, by decomposing $f(t)=\sum_{n\in\mathbb{Z}}\psi(t-\frac{n}{2})f(t)$ and $g(t)=
\sum_{n\in\mathbb{Z}}\psi(t-\frac{n}{2})g(t)$ with a suitable $\psi\in C^{\infty}_{0}(\mathbb{R})$
supported in (0,1), the general case for $f(t)$ and $g(t)$ follows from the considered particular case of
$f(t)$ and $g(t)$ supported in time in the interval $(0,1)$. Thus~\eqref{E:bilinearStrioneone}
 implies~\eqref{E:bilinearBourgain}.
\end{proof}

A similar proof for the gradient bilinear estimates should refer
to Anton~\cite{Ant06}.

\begin{lem}
Let $s\in\mathbb{R}$. The following statements are equivalent:

\noindent(1) For any $u_0\;,\;v_0\in L^2(M)$ satisfying
$$1_{\lambda\leq\sqrt{-\triangle}\leq2\lambda}u_0=u_0\;\;\;,
\;\;\;1_{\mu\leq\sqrt{-\triangle}\leq2\mu}v_0=v_0$$
one has
\begin{equation}\label{E:bilinearStritwotwo}
\|(\nabla
e^{it\triangle}u_0)\;e^{it\triangle}v_0\|_{L^2((0,1)_t\times
M)}\leq C\lambda({\rm min}(\lambda,\mu))^s
\|u_0\|_{L^2(M)}\|v_0\|_{L^2(M)}\end{equation}

\noindent(2)For any $b>\frac{1}{2}$ and any $f,g\in
X^{0,b}(\mathbb{R}\times M)$ satisfying
$$1_{\lambda\leq\sqrt{-\triangle}\leq2\lambda}f=f\;\;\;,
\;\;\;1_{\mu\leq\sqrt{-\triangle}\leq2\mu}g=g$$
one has
\begin{equation}\label{E:bilinearBourgaintwo}
\|(\nabla f)g\|_{L^2(\mathbb{R}\times
M)}\leq C\lambda({\rm
min}(\lambda,\mu))^s\|f\|_{X^{0,b}(\mathbb{R}\times M)}
\|g\|_{X^{0,b}(\mathbb{R}\times M)}
\end{equation}
\end{lem}

Denote by $S(t)=e^{it\triangle}$ the free evolution. Using the
Duhamel formula , we know that to solve~\eqref{E:Cauchyproblem} is equivalent to solve the
integral equation
$$u(t)=S(t)u_0-i\alpha\int_{0}^tS(t-\tau)\{|u(\tau)|^2u(\tau)\}d\tau$$

To deal with it , we need the following lemmas:
\begin{lem}\label{L:localexistenceone}
Let $b\;,\;s>0$ and let $u_0\in H^s(M)$. Then
\begin{equation}\label{E:linearBourgain}
\|S(t)u_0\|_{X^{s,b}_T}\lesssim
T^{\frac{1}{2}-b}\|u_0\|_{H^s}
\end{equation}
\end{lem}

\begin{lem}\label{L:localexistencetwo}
Let $0<b'<\frac{1}{2}$ and $0<b<1-b'$. Then for all $F\in
X^{s,-b'}_T(M)$,
\begin{equation}\label{E:nohomegenousBourgain}
\|\int_0^tS(t-\tau)F(\tau)d\tau\|_{X^{s,b}_T(M)}\lesssim T^{1-b-b'}\|F\|_{X^{s,-b'}_T(M)}
\end{equation}
\end{lem}

\begin{lem}\label{L:localexistencethree}
For $s>s_0$, there exists $(b,b')\in\mathbb{R}^2$, satisfying
\begin{equation}0<b'<\frac{1}{2}<b\;\;,\;\;b+b'<1,\end{equation} and $C>0$ such that for
every triple $(u_j),j=1,2,3$ in $X^{s,b}(\mathbb{R}\times M)$
\begin{equation}\label{E:nonlinearBourgain}
\|u_1u_2u_3\|_{X^{s,-b'}(\mathbb{R}\times M)}\leq C \prod_{j=1}^3\|u_j\|_{X^{s,b}
(\mathbb{R}\times M)}.\end{equation}
\end{lem}

Lemma~\ref{L:localexistenceone} is easy to see.
\begin{proof}
Let $\varepsilon>0$ and $\varphi\in C^{\infty}_{0}(\mathbb{R})\;,\;\varphi=1 $ on $(-T-\varepsilon,T+\varepsilon)$.
Then $\|S(t)u_0\|_{X^{s,b}_T}\leq \|\varphi(t)S(t)u_0\|_{X^{s,b}}\leq\|\varphi(t)u_0\|_{H^b(\mathbb{R},H^s(M))}
\leq cT^{\frac{1}{2}-b}\|u_0\|_{H^s(M)}$.\\
\end{proof}
The lemma~\ref{L:localexistencetwo} is due to Bourgain~\cite{Bou93}, we also refer to Ginibre~\cite{Gin96} for a simpler proof.

The proof of lemma~\ref{L:localexistencethree} will rely on the bilinear estimates~\eqref{E:bilinearBourgain}
and~\eqref{E:bilinearBourgaintwo}. However we will postpone this proof
and see how can we proof theorem~\ref{T:well-posed} by these there lemmas first.

\begin{proof}(of Theorem~\ref{T:well-posed}) To solve NLS equation is equivalent to solve the integral
equation with Dirichlet (or Neumann) boundary conditions
$$u(t)=S(t)u_0-i\alpha\int_0^t S(t-\tau)\{|u(\tau)|^2u(\tau)\}d\tau$$
We denote by $\Phi(u)$ by the left hand side of the equation.

Consider $(b,b')\in\mathbb{R}^2$ given by lemma~\ref{L:localexistencetwo} and let $R>0$ and $u_0\in H^s(M)$
such that $\|u_0\|_{H^s}\leq R$. We show that there exists $R'>0$ and $0<T<1$ depending on $R$ such that
 $\Phi$ is a contracting map from the ball $B(0,R')\subset X^{s,b}_T(M)$ onto itself.

From the linear estimate~\eqref{E:linearBourgain} we know that
$\|S(t)u_0\|_{X^{s,b}_1(M)}\leq c\|u_0\|_{H^s}$. From the
definition of $X^{s,b}_T$ spaces we know that $T_1<T_2$ implies
$X^{s,b}_{T_2}\subset X^{s,b}_{T_1}$.
Therefore for $T<1$, $\|S(t)u_0\|_{X^{s,b}_T(M)}\leq c_0\|u_0\|_{H^s}$.

Define $R'=2c_0R$. From estimates~\eqref{E:nohomegenousBourgain} , we obtain for $T<1$,
$$\|\Phi(u)\|_{X^{s,b}_T(M)}\leq c_0\|u_0\|_{H^s}+c_1T^{1-b-b'}\|u\overline{u}u\|_{X^{s,-b'}_T(M)}$$
Combine this with~\eqref{E:nonlinearBourgain} gives
$$\|\Phi(u)\|_{X^{s,b}_T(M)}\leq c_0\|u_0\|_{H^s}+c_2T^{1-b-b'}\|u\|^3_{X^{s,b}_T(M)} .$$

Taking $T<1$ such
that $T^{1-b-b'}c_2R'^3\leq c_0 R$, we ensure $\Phi:B(0,R')\subset X^{s,b}_T\rightarrow B(0,R')\subset X^{s,b}_T.$
In addition $\Phi$ is a contraction, let $u_1,u_2\in B(0,R')\subset X^{s,b}_T$, then
$$\|\Phi(u_1)-\Phi(u_2)\|_{X^{s,b}_T(M)}\leq c_2T^{1-b-b'}\||u_1|^2u_2-|u_2|^2u_1\|_{X^{s,b}_T(M)}.$$

Using the decomposition
$|u_1|^2u_1-|u_2|^2u_2=u_1^2(\overline{u}_1-\overline{u}_2)+\overline{u}_2(u_1-u_2)(u_1+u_2)$
,~\eqref{E:nohomegenousBourgain} and~\eqref{E:nonlinearBourgain} , we get
$$\|\Phi(u_1)-\Phi(u_2)\|_{X^{s,b}_T(M)}\leq c_3T^{1-b-b'}R'^2\|u_1-u_2\|_{X^{s,b}_T(M)}.$$
 By choosing $T<1$
sufficient small , we know $\Phi$ is a contraction. Thus there exists an uniqueness $u\in X^{s,b}_T(M)$ such
that $\Phi(u)=u$. Since $b>\frac{1}{2}\;,\;u\in C((-T,T),H^s(M))$.
The flow $u_0\in B(0,R)\subset H^s(M)\rightarrow u\in X^{s,b}_T(M)$ is Lipschitz. For if $u\;,\;v$ are two
solutions with initial data $u_0\;,\;v_0$, we have as above
$$\|u-v\|_{X^{s,b}_T}\leq c\|u_0-v_0\|_{H^s}+c_3T^{1-b-b'}R'^2\|u-v\|_{X^{s,b}_T}.$$

 By choosing $T$ small enough , we have
$$\|u-v\|_{X^{s,b}_T}\leq c\|u_0-v_0\|_{H^s}$$

\end{proof}

\subsection{Nonlinear Analysis}

Now we only owe to prove Lemma~\ref{L:localexistencethree}. We will use a decomposition of
the spectrum of functions $u_j\in X^{s,b}
(\mathbb{R}\times M)$.

The duality argument leads to the following equivalence: $u\in X^{s,b}(\mathbb{R}\times M)\;,\;\Leftrightarrow$
 for all $u_0\in X^{\infty,\infty}(\mathbb{R}\times M)=\cap_{s>0,b\in\mathbb{R}}X^{s,b}(\mathbb{R}\times M)$
 we have
$$|<u,u_0>|\leq c\|u_0\|_{X^{-s,-b}(\mathbb{R}\times M)}$$
 where $<,>$ denote the bracket
 pairing $\mathcal{S}'$ and $\mathcal{S}$. Thus~\eqref{E:nonlinearBourgain} is implied by
 \begin{equation}
|\int_{\mathbb{R}}\int_{M}u_0u_1u_2u_3dxdt|\leq
c\prod_{j=1}^3\|u_j\|_{X^{s,b}(\mathbb{R}\times
M)}\|u_0\|_{X^{-s,b'(\mathbb{R}\times M)}}
 \end{equation} holding for all $u_0\in X^{\infty,\infty}(\mathbb{R}\times M)$. We will prove a similar result
 for spectrally localized functions and then sum over all frequencies.

 For $j\in\{0,1,2,3\}$ and $N_j\in 2^{\mathbb{N}}$. We denote by $u_{jN_j}=1_{\sqrt{-\triangle}\in[N_j,2N_j]}u_j$.
 Using the definition of $X^{s,b}(\mathbb{R}\times M)$ spaces the following equivalence holds
 \begin{equation}\label{E:Bourgainequivalence}
 \|u_j\|^2_{X^{s,b}(\mathbb{R}\times M)}\cong\sum_{N_j\in 2^{\mathbb{N}}}\|u_{jN_j}\|
 ^2_{X^{s,b}(\mathbb{R}\times M)}
 \cong \sum_{N_j\in 2^{\mathbb{N}}} N^{2s}_j\|u_{jN_j}\|^2_{X^{0,b}(\mathbb{R}\times M)}.\end{equation}
 We denote by
 $\underline{N}=(N_0,N_1,N_2,N_3)$ the quadruple of $2^{n}$ numbers, $n\in\mathbb{N}$. Also
$$I(\underline{N})=\int_{\mathbb{R}\times M}\prod_{i=0}^3u_{jN_j}dxdt$$

 In order to prove Lemma~\ref{L:localexistencethree}. We need the two
 estimates about $I(\underline{N})$ in the following lemma. The proof of first estimate is
 standard by using~\eqref{E:bilinearBourgain} , while the second estimate in this lemma
 with Dirichlet boundary condition was proved
 by Anton~\cite{Ant07} using~\eqref{E:bilinearBourgaintwo}.
 The same argument works for either Dirichlet or Neumann condition.  For the completeness
 and benefit of readers to understand how the bilinear estimates and gradient bilinear estimates
 working in nonlinear analysis, we include its proof here .

  We also need the fact that
 \begin{equation}\label{E:imbeddingBourgain}
  \|f\|_{L^4(\mathbb{R},L^2(M))}\leq\|f\|_{X^{0,\frac{1}{4}}(\mathbb{R}\times M)}.
 \end{equation} This is due to conservation of $L^2$ norm by the linear Schr\"{o}dinger
 flow and Sobolev embedding $H^{\frac{1}{4}}(\mathbb{R})\hookrightarrow
 L^4(\mathbb{R})$, thus

$$\|f\|_{L^4(\mathbb{R},L^2(M))}=\|e^{it\triangle}f\|_{L^4(\mathbb{R},L^2(M))}\leq\|e^{it\triangle}f\|_
 {H^{\frac{1}{4}}(\mathbb{R}\times L^2(M))}=\|f\|_{X^{0,\frac{1}{4}}(\mathbb{R}\times M)}.$$

\begin{lem}\label{L:twoestimates}
If~\eqref{E:bilinearStrioneone} and~\eqref{E:bilinearStritwotwo} hold for $s>s_0$, then for all
$s'>s_0$ there exists $0<b'<\frac{1}{2}\;,c>0$ such that,
assuming $N_3\leq N_2\leq N_1$, the following estimates hold:
\begin{equation}\label{E:nonlineargoalone}
|I(\underline{N})|\leq
c(N_2N_3)^{s'}\prod_{j=0}^3\|u_{jN_j}\|_{X^{0,b'}(\mathbb{R}\times
M)}
\end{equation}
\begin{equation}\label{E:nonlineargoaltwo}
|I(\underline{N})|\leq
c(\frac{N_1}{N_0})^2(N_2N_3)^{s'}\prod_{j=0}^3\|u_{jN_j}\|_{X^{0,b'}(\mathbb{R}\times
M)}
\end{equation}
\end{lem}
\begin{proof}
Use Holder inequality, we get
\begin{align}\label{E:nonlinearone}
|I(\underline{N})|&\leq\|u_{3N_3}\|_{L^4(L_x^{\infty})}\|u_{2N_2}\|_{L^4(L_x^{\infty})}
\|u_{1N_1}\|_{L^4(L^2_x)}\|u_{0N_0}\|_{L^4(L^2_x)}\notag\\
&\leq
c(N_2N_3)^{1+\varepsilon}\prod_{j=0}^3\|u_{jN_j}\|_{L^4(L_x^2)}\notag\\
&\leq c
(N_2N_3)^{1+\varepsilon}\prod_{j=0}^3\|u_{jN_j}\|_{X^{0,\frac{1}{4}}(\mathbb{R}\times
M)}
\end{align}
In the second inequality, we use Sobolev embedding $\|u_{N_j}\|_{L^{\infty}(M)}\leq c
N_j^{1+\varepsilon}\|u_{N_j}\|_{L^2(M)}$. The third inequality came from~\eqref{E:imbeddingBourgain} .

Use Cauchy inequality and~\eqref{E:bilinearBourgain} (which is implied by~\eqref{E:bilinearStrioneone} ),
we obtain that for any $b_0>\frac{1}{2}$ there exists
$c_0>0$ such that
\begin{align}\label{E:nonlineartwo}
|I(\underline{N})|&\leq\|u_{0N_0}u_{2N_2}\|_{L^2(\mathbb{R}\times
M)}\|u_{1N_1}u_{3N_3}\|_{L^2(\mathbb{R}\times M)}
\notag\\
&\leq
c_1(N_2N_3)^{s_0}\prod_{j=0}^3\|u_{jN_j}\|_{X^{0,b_0}(\mathbb{R}\times
M)}
\end{align}

We need further decomposition $u_{jN_j}=\sum_{K_j}u_{jN_jK_j}$ for
interpolation, where $u_{jN_jK_j}=1_{K_j\leq \langle
i\partial_t+\triangle \rangle\leq 2K_j}u_{jN_j}$ and the sum is
taken over $2^n$ numbers , for
$n\in\mathbb{N}:K_j\in 2^{\mathbb{N}}$. Let us denote $I(\underline{N},\underline{K})=
\int_{\mathbb{R}\times M}\prod_{j=0}^3u_{jN_jK_j}$. Estimates~\eqref{E:nonlinearone}
and~\eqref{E:nonlineartwo} give
$$|I(\underline{N},\underline{K})|\leq c(N_2N_3)^{\alpha}(\prod_{j=0}^3K_j)^{\beta}\prod_{j=0}^3
\|u_{jN_jK_j}\|_{L^2(\mathbb{R}\times M)}$$
where $(\alpha,\beta)$ equals $(1+\varepsilon,\frac{1}{4})$ or
$(s_0,b_0)$. For $s_0<s<1$ we can choose $\varepsilon>0\;,\;b_0>\frac{1}{2}$ and $0<b_1<\frac{1}{2}$ such
that by interpolation we have the same estimates for $(\alpha,\beta)=(s',b_1)$.

   Taking $b'\in(b_1,\frac{1}{2})$, this reads
$$|I(\underline{N},\underline{K})|\leq c(N_2N_3)^{s'}\prod_{j=0}^3K_j^{b_1-b'}\|u_{jN_jK_j}\|_{X^{0,b'}
(\mathbb{R}\times M)}.$$
Summing up over $\underline{K}\in
(2^{\mathbb{N}})^4$, by geometric series and using Cauchy
Schwartz, we obtain
$$|I(\underline{N})|\leq c(N_2N_3)^{s'}\prod_{j=0}^3\|u_{jN_j}\|_{X^{0,b'}(\mathbb{R}\times M)}$$
 which conclude the
proof of~\eqref{E:nonlineargoalone}.

For the proof of~\eqref{E:nonlineargoaltwo} , we start with Green formula:
$$\int_{M}\triangle fg-f\triangle gdx=\int_{\partial M}\frac{\partial f}{\partial \upsilon}g-f\frac{\partial g}
{\partial \upsilon}d\sigma $$

If ${e_k}$ are eigenfunctions of the Dicichlet(or Neumann)
Laplacian associated with eigenvalues $\lambda_k^2$. The
$u_{0N_0}=\sum_{\lambda_k\sim N_0}c_ke_k$, where
$c_k=(u_{0N_0},e_k)$. We write
$$u_{0N_0}=-\frac{\Delta}{N_0^2}\sum_{\lambda_k\sim N_0}c_k(\frac{N_0}{\lambda_k})^2e_k.$$
Define $Tu_{0N_0}=\sum_{\lambda_k\sim N_0}c_k(\frac{N_0}{\lambda_k})^2e_k$ and $Vu_{0N_0}=
\sum_{\lambda_k\sim N_0}c_k(\frac{\lambda_k}{N_0})^2e_k.$ Then we have $TVu_{0N_0}=VTu_{0N_0}=u_{0N_0}$ and
$\|Tu_{0N_0}\|_{H^s}\sim \|u_{0N_0}\|_{H^s}$ for all $s$. Use this notation
$u_{0N_0}=-\frac{\Delta}{(N_0)^2}Tu_{0N_0}.$
 Apply it to green formula and using $u_{jN_j}|_{\partial M}=0$
(or $N_x\cdot\nabla u|_{\partial M}=0$),
we obtain
$$I(\underline{N})=\frac{1}{N_0^2}\int_{\mathbb{R}\times M}Tu_0N_0\Delta(u_{1N_1}u_{2N_2}u_{3N_3})$$

By Leibniz's law, we have to deal with summation of terms of the forms
$$\frac{1}{N_0^2}J_{11}(\underline{N})=\frac{1}{N_0^2}\int_{\mathbb{R}\times M}Tu_{0N_0}(\Delta u_{1N_1})u_{2N_2}u_{3N_3}$$
and
$$\frac{1}{N_0^2}J_{12}(\underline{N})=\frac{1}{N_0^2}\int_{\mathbb{R}\times M}Tu_{0N_0}(\nabla u_{1N_1})
 (\nabla u_{2N_2})u_{3N_3}.$$
 As we will see soon, they are always the largest terms in each sum. Use
 $\triangle u_{2N_2}$ we get $J_{11}(\underline{N})=-N_1^2\int_{\mathbb{R}\times M}Tu_{0N_0}Vu_{1N_1}
 u_{2N_2}u_{3N_3}$. Thus by~\eqref{E:nonlineargoalone} and $\|u_{jN_j}\|_{H^s}\sim\|Tu_{jN_j}\|_{H^s}\sim\|Vu_{jN_j}\|_{H^s}$, we have
$$\frac{1}{N_0^2}|J_{11}(\underline{N})|\leq c \frac{N_1^2}{N_0^2}(N_2N_3)^{s'}\prod_{j=0}^3\|u_{jN_j}\|_
 {X^{0,b'}}(\mathbb{R}\times M).$$

To estimates $J_{12}(\underline{N})$, we note that $\|\nabla
u_{jN_j}\|_{L^2(M)}\leq cN_j\|u_{jN_j}\|_{L^2(M)}$. Use the same
process as in the proof of~\eqref{E:nonlineargoalone}  ,
then~\eqref{E:nonlinearone} and~\eqref{E:nonlineartwo} correspond to
$$|J_{12}(\underline{N})|\leq c(N_1N_2)(N_2N_3)^{1+\varepsilon}\prod_{j=0}^3\|u_{jN_j}\|_{X^{0,\frac{1}{4}}
(\mathbb{R})\times M}$$
 and
$$|J_{12}(\underline{N})|\leq c(N_1N_2)(N_2N_3)^{s_0}\prod_{j=0}^3\|u_{jN_j}\|_{X^{0,b_0}(\mathbb{R})\times M}.$$

In fact, we just got an additional term $N_1N_2$ in these new estimates.  Therefore the interpolation argument
leads to
$$\frac{1}{N_0^2}|J_{12}(\underline{N})|\leq c\frac{N_1N_2}{N_0^2}(N_2N_3)^{s'}\prod_{j=0}^3
\|u_{jN_j}\|_{X^{0,b'}}(\mathbb{R}\times M).$$
Since $N_1N_2\leq N_1^2$, we are done.
\end{proof}

Now we can use Lemma~\ref{L:twoestimates} to prove Lemma~\ref{L:localexistencethree} .
\begin{proof}(Proof of Lemma~\eqref{L:localexistencethree} )\\
Our goal is to prove~\eqref{E:nonlinearBourgain} . Use the same notation as above, we consider $I(\underline{N})=
\int_{\mathbb{R}\times M}\prod_{i=0}^3u_{jN_j}dxdt$. Without loss of generality, we may assume $N_3\leq N_2\leq N_1$.

Let $\frac{2}{3}<s'<s$. Using~\eqref{E:nonlineargoaltwo} in Lemma~\ref{L:twoestimates} and~\eqref{E:Bourgainequivalence} , we have
$$|\sum_{N_0<cN_1}I(\underline{N})|\leq c\sum_{N_0<cN_1}(N_2N_3)^{s'-s}(\frac{N_0}{N_1})^s\|u_{0N_0}\|_{X^{-s,b'}
\mathbb{R}\times M}\prod_{j=1}^3\|u_{jN_j}\|_{X^{s,b'}(\mathbb{R}\times M)}. $$

Using Cauchy Schwartz inequality and~\eqref{E:Bourgainequivalence}, we have
 $$|\sum_{N_0<cN_1}I(\underline{N})|\leq c\|u_2\|_{X^{s,b'}(\mathbb{R}\times M)}
 \|u_3\|_{X^{s,b'}(\mathbb{R}\times M)}\sum_{N_0\leq CN_1}(\frac{N_0}{N_1})^s\alpha(N_0)\beta(N_1).$$
 where $\alpha(N_0)=\|u_{0N_0}\|_{X^{-s,b'}(\mathbb{R}\times M)}$ and $\beta(N_1)=\|u_{1N_1}\|_{X^{s,b'}
 (\mathbb{R}\times M)}$. Thus we have
$$\sum_{N_0}\alpha(N_0)^2\cong \|u_0\|^2_{X^{-s,b'}}\;,\;\sum_{N_1}\beta(N_1)^2\cong \|u_1\|^2_{X^{s,b'}}.$$

 Since $N_0\;,\;N_1$ are both dyadic numbers, we write $N_1=2^lN_0$ and $N_0\geq N(l)={\rm max}(1,2^{-l})$, where $l$
 is an integer, $l\geq -l_0$ for some $l_0\in\mathbb{N}$ depending on $c$. Thus
 \begin{align}
\sum_{N_0<cN_1}&(\frac{N_0}{N_1})^s\alpha(N_0)\beta(N_1)=\sum_{l\geq-l_0}\sum_{N_0\geq N(l)}2^{-sl}
\alpha(N_0)\beta(2^lN_0)\notag\\
&\leq
\sum_{l>-l_0}2^{-sl}(\sum_{N_0}\alpha(N_0)^2)^{\frac{1}{2}}(\sum_{N_0>N(l)}\beta(2^lN_0)^2)^{\frac{1}{2}}\notag\\
&\leq c\|u_0\|_{X^{-s,b'}(\mathbb{R}\times
M)}\|u_1\|_{X^{s,b'}(\mathbb{R}\times M)}\notag
 \end{align} Since $\|u\|_{X^{s,b'}}\leq\|u\|_{X^{s,b}}$ for $b'<b$, we
 conclude that
$$|\sum_{N_0<cN_1}I(\underline{N})|\leq c\|u_0\|_{X^{-s,b'}}\prod_{j=1}^3\|u_j\|^{X^{s,b}}.$$

For $N_0\geq cN_1$, we use~\eqref{E:nonlineargoaltwo} of Lemma~\ref{L:twoestimates} to get:
$$|\sum_{N_0\geq cN_1}I(\underline{N})|\leq c\sum_{N_0\geq cN_1}(N_2N_3)^{s'-s}(\frac{N_1}{N_0})^{2-s}\|u_{0N_0}\|_{X^{-s,b'}
\mathbb{R}\times
M}\prod_{j=1}^3\|u_{jN_j}\|_{X^{s,b'}(\mathbb{R}\times M)}.$$ This is just an exchange the role of $N_0$ and $N_1$
 in the previous argument. Thus we obtain again
$$|\sum_{N_0\geq cN_1}I(\underline{N})|\leq c \|u_0\|_{X^{-s,b'}(\mathbb{R}\times M)}\|u_1\|_{X^{s,b'}
 (\mathbb{R}\times M)}\|u_2\|_{X^{s,b'}(\mathbb{R}\times M)}\|u_3\|_{X^{s,b'}(\mathbb{R}\times M)}$$

\end{proof}

\end{document}